\documentclass[12pt,reqno]{amsart}

\usepackage[margin=1in]{geometry}
\usepackage{amsmath, amssymb, amsthm}
\usepackage{mathtools} \usepackage{xcolor} \usepackage{bm} \usepackage{float}
\usepackage[colorlinks=true, linkcolor=black, citecolor=black,urlcolor=black, filecolor=black]{hyperref} \usepackage{tikz}

\definecolor{cbBlue}{RGB}{0,114,178}
\definecolor{cbVermillion}{RGB}{213,94,0}
\definecolor{cbOrange}{RGB}{230,159,0}
\definecolor{cbSkyBlue}{RGB}{86,180,233}
\definecolor{cbGreen}{RGB}{0,158,115}
\definecolor{cbYellow}{RGB}{240,228,66}
\definecolor{cbPurple}{RGB}{204,121,167}

\theoremstyle{plain}
\newtheorem{theorem}{Theorem}[section]
\newtheorem{proposition}[theorem]{Proposition}
\newtheorem{lemma}[theorem]{Lemma}
\newtheorem{corollary}[theorem]{Corollary}

\theoremstyle{definition}
\newtheorem{definition}[theorem]{Definition}
\newtheorem{example}[theorem]{Example}
\newtheorem{remark}[theorem]{Remark}

\newcommand{\Z}{\mathbb{Z}}
\newcommand{\Q}{\mathbb{Q}}
\newcommand{\N}{\mathbb{N}}
\newcommand{\R}{\mathbb{R}}

\newcommand{\M}{\mathcal{M}}
\newcommand{\hyp}{\Delta_{k,n}}
\newcommand{\Tri}{\mathcal{T}_{k,n}}
\newcommand{\dosp}{\mathsf{S}}
\newcommand{\triv}{\mathbf{1}}

\DeclareMathOperator{\des}{des}
\DeclareMathOperator{\typ}{type}
\DeclareMathOperator{\comp}{comp}
\DeclareMathOperator{\ccomp}{ccomp}
\DeclareMathOperator{\Hyp}{Hyp}
\DeclareMathOperator{\base}{base}
\DeclareMathOperator{\cEq}{cEq}
\DeclareMathOperator{\Stab}{Stab}

\DeclareMathOperator{\cDes}{cDes}

\DeclareMathOperator{\wind}{wind}
\DeclareMathOperator{\Ehr}{Ehr}

\DeclareMathOperator{\hDOSP}{hDOSP}

\DeclareMathOperator{\col}{color}

\DeclareMathOperator{\Ind}{Ind}

\DeclareMathOperator{\Hom}{Hom}

\DeclareMathOperator{\tr}{tr}

\DeclareMathOperator{\cdes}{cdes}
\DeclareMathOperator{\ceq}{ceq}
\DeclareMathOperator{\cmaj}{cmaj}
\DeclareMathOperator{\content}{content}
\DeclareMathOperator{\ch}{ch}

\begin{document}

\title{A flag refinement of the $ h^* $-formula for the hypersimplices}
\author[Y. Jiang]{Yuhan Jiang}
\email{jyh@math.berkeley.edu}
\address{University of California, Berkeley}

\begin{abstract}
	The hypersimplex $ \hyp $
	is the convex hull of all 0/1-vectors of length $ n $ with coordinate sum $ k $.
	Early conjectured, and Kim proved a combinatorial formula for the $ h^* $-polynomial of $ \hyp $ in terms of hypersimplicial decorated ordered set partitions.
In this paper, we refine the Early--Kim formula to the flag $h$-numbers of the alcove triangulation of $\hyp$. We also interpret its flag $f$-numbers in terms of Young permutation modules. When $\gcd(k,n)=1$, these numbers count Young subgroup orbits on hypersimplicial decorated ordered set partitions.
\end{abstract}

\maketitle

\section{Introduction}
Throughout, $0<k<n$. The hypersimplex $\hyp=\{x\in[0,1]^n:\sum_i x_i=k\}$ is an $(n-1)$-dimensional polytope with $\binom{n}{k}$ vertices. The \emph{Ehrhart polynomial} $\Ehr(P,t)$ of a lattice polytope $P$ satisfies
\[
	\Ehr(P,t) = |tP \cap \Z^n|, \qquad \text{ for }t \in \N.
\]
The \emph{ $ h^* $-polynomial} of $ P $ is defined by the generating function
\[
	\sum_{t \ge 0} \Ehr(P,t) z^t = \frac{h^*(P,z)}{(1-z)^{\dim(P)+1}}.
\]
Early conjectured a combinatorial formula for the $h^*$-polynomial of $ \hyp $ in terms of \emph{hypersimplicial
  decorated ordered set partitions}~\cite{Early2017}, which was later proved by Kim~\cite{Kim2020}.
\begin{definition}[\cite{EarlyPlates}]\label{def:dosp}
	Let $(S_1, \dots, S_p)$ be an ordered set partition of $ [n] $ and $(s_1, \dots, s_p)$ be a composition of $k$.
	There is an equivalence relation on these pairs, defined by
	\[
		((S_1)_{s_1},(S_2)_{s_2},\dots,(S_p)_{s_p}) \sim
		((S_2)_{s_2},\dots,(S_p)_{s_p},(S_1)_{s_1}).
	\]
	A \emph{decorated ordered set partition}
	\begin{equation}\label{eq:dosp}
		\dosp = [((S_1)_{s_1},\dots,(S_p)_{s_p})]
	\end{equation}
	is an equivalence class of such pairs.
A decorated ordered set partition is \emph{hypersimplicial} if $1\le s_i\le |S_i|-1$ for all $i$. Denote the set of such partitions of type $(k,n)$ by $\hDOSP(k,n)$. Place the blocks $S_1,\ldots,S_p$ clockwise on a circle with $k$ positions, with $s_i$ steps from $S_i$ to $S_{i+1}$.
Thus $s_i-1$ positions between these blocks are empty. The \emph{winding number} $\wind(\dosp)$ counts the turns made while tracing $1\to2\to\cdots\to n\to1$ clockwise, staying in place when consecutive labels belong to the same block.
\end{definition}
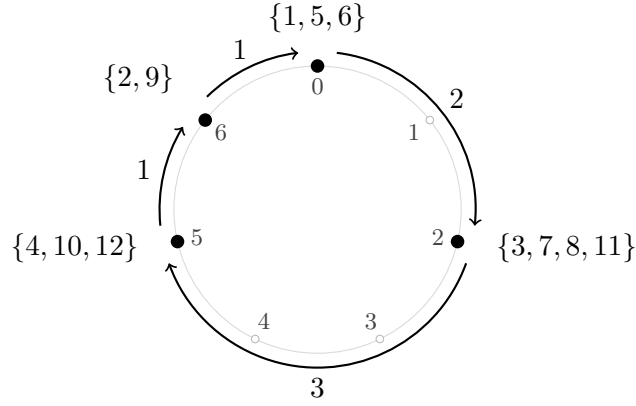
\begin{figure}
	\centering
	\begin{tikzpicture}[scale=2]
	\draw[gray!30] (0,0) circle (1);
	\foreach \p in {1,...,7}{
			\node[circle,draw=gray!55,fill=white,inner sep=1pt] at ({90-360*\p/7}:1) {};
			\node[black!75,font=\scriptsize] at ({90-360*(\p)/7}:0.86) {\p};}
	\foreach \p in {0,2,5,6}{\node[circle,fill=black,inner sep=1.8pt] at ({90-360*\p/7}:1) {};}
	\node[font=\small,anchor=south]      at ({90}:1.16)        {$\{1,5,6\}$};
	\node[font=\small,anchor=west]       at ({90-720/7}:1.20)  {$\{3,7,8,11\}$};
	\node[font=\small,anchor=east]       at ({90-1800/7}:1.20) {$\{4,10,12\}$};
	\node[font=\small,anchor=south east] at ({90-2160/7}:1.18) {$\{2,9\}$};
	\draw[->,thick] ({90-7}:1.10) arc ({90-7}:{90-720/7+7}:1.10);\node[font=\small] at ({90-360/7}:1.24) {$2$};
	\draw[->,thick] ({90-720/7-7}:1.10) arc ({90-720/7-7}:{90-1800/7+7}:1.10);\node[font=\small] at ({90-1260/7}:1.24) {$3$};
	\draw[->,thick] ({90-1800/7-7}:1.10) arc ({90-1800/7-7}:{90-2160/7+7}:1.10);\node[font=\small] at ({90-1980/7}:1.24) {$1$};
	\draw[->,thick] ({90-2160/7-7}:1.10) arc ({90-2160/7-7}:{90-360+7}:1.10);\node[font=\small] at ({90-2340/7}:1.24) {$1$};
	\end{tikzpicture}
	\caption{The decorated ordered set partition $
			[(\{3,7,8,11\}_3,\{4,10,12\}_1,\{2,9\}_1,\{1,5,6\}_2,)] $ has type $ (7,12) $ and winding number $ 5 $.}\label{fig:dosp}
\end{figure}

\begin{remark}
	Let $(S_1,\dots,S_p)$ be an ordered set partition of $[n]$, and $(s_1,\dots,s_p)$ be a composition of $k$.
	For a set $S$, let $x_S = \sum_{i \in S} x_i$.
	Derksen and Fink~\cite{DerksenFink} use chain inequalities in their description of matroid base polytopes. For the ordered partition and decorations above, consider the region
	\begin{equation}
		\begin{aligned}
			x_{S_1}          & \leq s_1 \\
			x_{S_1 \cup S_2} & \leq s_1+s_2 \\
			                 & \vdots \\
			x_{[n]}          & = k.
		\end{aligned}
	\end{equation}
        In the matroid case, Derksen and Fink~\cite[Theorem`5.4]{DerksenFink} intersect the corresponding regions with $\hyp$. The indicator functions of these intersections form a $\Z$-basis of the group generated by the indicator functions of the rank-$k$ matroid base polytopes on $[n]$.
        On $x_{[n]}=k$, each displayed upper bound is a lower bound on teh complementary suffix. Thus the indicator of the displayed region is Ocneanu's plate for the reversed sequence of decorated blocks in the convention of ~\cite{EarlyPlates,Early2016thesis}.
\end{remark}

\begin{theorem}[\cite{Early2017,Kim2020}]\label{thm:earlykim}
	The $ h^* $ -polynomial of the hypersimplex $ \hyp $ is given by
	\[
		h^*(\hyp,z) = \sum_{\dosp \in \hDOSP(k,n)} z^{\wind(\dosp)}.
	\]
\end{theorem}

\begin{figure}
	\centering
	\begin{tikzpicture}[scale=.8,every node/.style={font=\small}]
	\coordinate (A) at (0,1);
	\coordinate (B) at (2.5,0);
	\coordinate (C) at (4,1);
	\coordinate (D) at (1.5,2);

	\coordinate (Apex) at (2,4);
	\coordinate (Bot) at (2,-2);

	\draw[thick] (A) -- (B) -- (C);
	\draw[thick,dotted] (C) -- (D) --(A);

	\draw[thick] (Apex) -- (A);
	\draw[thick] (Apex) -- (B);
	\draw[thick] (Apex) -- (C);
	\draw[thick,dotted] (Apex) -- (D);
	\draw[thick] (Bot) -- (A);
	\draw[thick] (Bot) -- (B);
	\draw[thick] (Bot) -- (C);
	\draw[thick,dotted] (Bot) -- (D);

	\node[below left] at (A) {0110};
	\node[below right] at (B) {1010};
	\node[above right] at (C) {1001};
	\node[above left] at (D) {0101};

	\node[above] at (Apex) {1100};
	\node[below] at (Bot) {0011};
	\end{tikzpicture}
	\caption{The hypersimplex $ \Delta_{2,4} $.}\label{fig:d24}
\end{figure}
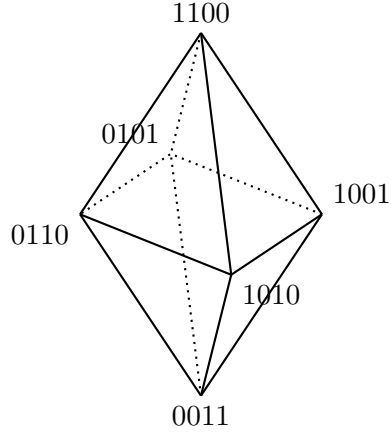

\begin{figure}
	\centering
	\begin{tikzpicture}[scale=.8, every node/.style={fill=white, inner sep=1pt, font=\small},
	arc/.style={thick, ->, >=Stealth}]
	\draw[thick] (0,0) circle (2);
	\node (A) at (90:2) { $ \{1,2,3,4\} $ };
	\node at (0,-2.5) {winding number $ 0 $ };
	\end{tikzpicture}
	\begin{tikzpicture}[scale=.8,every node/.style={fill=white, inner sep=1pt, font=\small},
	arc/.style={thick, ->, >=Stealth}]
	\draw[thick] (0,0) circle (2);
	\node (A) at (270:2) { $ \{1,2\} $ };
	\node (B) at (90:2) { $ \{3,4\} $ };
	\node at (0,-2.5) {winding number $ 1 $ };
	\end{tikzpicture}
	\begin{tikzpicture}[scale=.8,every node/.style={fill=white, inner sep=1pt, font=\small},
	arc/.style={thick, ->, >=Stealth}]
	\draw[thick] (0,0) circle (2);
	\node (A) at (270:2) { $ \{1,4\} $ };
	\node (B) at (90:2) { $ \{2,3\} $ };
	\node at (0,-2.5) {winding number $ 1 $ };
	\end{tikzpicture}
	\begin{tikzpicture}[scale=.8,every node/.style={fill=white, inner sep=1pt, font=\small},
	arc/.style={thick, ->, >=Stealth}]
	\draw[thick] (0,0) circle (2);
	\node (A) at (270:2) { $ \{1,3\} $ };
	\node (B) at (90:2) { $ \{2,4\} $ };
	\node at (0,-2.5) {winding number $ 2 $ };
	\end{tikzpicture}
	\caption{ $ \hDOSP(2,4) $.}\label{fig:dosp24}
\end{figure}

\begin{example}\label{ex:d24}
	The hypersimplex $ \Delta_{2,4} $
	is the octahedron shown in Figure~\ref{fig:d24}.
	The hypersimplicial decorated ordered set partitions of type $ (2,4) $
	are shown in Figure~\ref{fig:dosp24}, and we have
	\[
		h^*(\Delta_{2,4},z) = 1 + 2z + z^2.
	\]
\end{example}

\subsection{Balanced complexes and flag numbers}\label{ss:balanced-flag}
There is a \emph{unimodular triangulation} of the hypersimplex $\hyp$ by \emph{alcoves}, studied by Postnikov and Lam
in~\cite{alcove1}. The \emph{alcove triangulation complex} of $\hyp$, denoted $\Tri$, is the simplicial complex with
vertex set $\binom{[n]}{k}$, and maximal faces called \emph{alcoves} (See Section~\ref{sec:triangulation}). Since
$\Tri$ is unimodular, Stanley's result in~\cite{stanley80} implies that \[h^*_d(\hyp) = h_d(\Tri).\] Our main result gives a refinement of Theorem~\ref{thm:earlykim} by flag numbers introduced by Stanley in the study
   of balanced complexes~\cite{Stanley1979}.

Let $\Delta $ be a simplicial complex on a vertex set $ V $. A face with $ i+1 $ elements is an \emph{ $ i $-face} and
has $ \dim F=i $; if $ \delta=\max\{\dim F:F\in\Delta\} $ then $ \Delta $ is a \emph{$ \delta $-complex} and $
\dim\Delta=\delta $, and $ \Delta $ is \emph{pure} if every maximal face is a $ \delta $-face~\cite[\S1]{Stanley1979}.

\begin{definition}[{\cite[\S2]{Stanley1979}}]\label{def:balanced}
	A \emph{balanced complex} of type $ (a_1,\dots,a_m) $ is a pair $ (\Delta,\pi) $ such that $ \Delta $ is a pure $ \delta $
	-complex on a vertex set $ V $ and $\pi=(V_1,\ldots, V_m)$ is an ordered set partition of $ V $
	such that for every maximal face $ F\in\Delta $ and every $ i\in[m] $, we have $ |F\cap V_i|=a_i $. Hence $
		a_1+\cdots+a_m=\delta+1 $.
	A balanced complex of type $ (1,1,\dots,1) $
	is \emph{completely balanced}.
\end{definition}

\begin{definition}[{\cite[\S3]{Stanley1979}}]\label{def:flag}
	Let $ (\Delta,\pi) $ be completely balanced with $\pi = (V_1, \ldots,V_m)$. The \emph{color} of a vertex $ x\in V_i $ is $ \col(x)=i $, and the \emph{type} of a face $ F\in\Delta $ is
	\begin{equation}
		\typ(F)=\{\col(x):x\in F\}\subseteq[m].
	\end{equation}
	For $ S\subseteq[m] $ the \emph{flag $ f $-number} and the \emph{flag $ h $-number} of $ \Delta $ are
	\begin{equation}\label{eq:flagnumbers}
		f_S(\Delta)=\#\{F\in\Delta:\typ(F)=S\},\qquad
		h_S(\Delta)=\sum_{T\subseteq S}(-1)^{|S\setminus T|}f_T(\Delta).
	\end{equation}
\end{definition}

\begin{lemma}[Eq.6,\cite{Stanley1979}]
	If $\Delta$ is a completely balanced complex, then
	\begin{equation}\label{eq:flagh}
		h_d(\Delta)=\sum_{|S|=d} h_S(\Delta).
	\end{equation}
\end{lemma}

\subsection{Cyclic compositions and cyclic words}
A \emph{composition} of $n$ with $p$ parts is a vector
$\alpha=(\alpha_1,\ldots,\alpha_p)$ of positive integers with sum $n$. For a \emph{weak composition} we allow zero parts.
We write $\alpha\models n$ and use the cyclic rotations as an equivalence relation,
\[
  (\alpha_1,\ldots,\alpha_p)\sim
  \alpha^{(i)}:=(\alpha_i,\ldots,\alpha_p,\alpha_1,\ldots,\alpha_{i-1}),
  \qquad 1\le i\le p.
\]
The equivalence class under this relation is called a \emph{cyclic composition}, denoted $ [\alpha] $.

\begin{definition}
  For a nonempty subset $S\subseteq[n]$ and $c\in S$, let $\comp^{(c)}(S)$ be the composition of $n$ obtained by listing the clockwise arc lengths between consecutive points of $S$, starting at $c$. In particular, $\comp^{(c)}(\{c\})=(n)$. Changing the starting point cyclically permutes the parts. Denote the resulting cyclic class by $\ccomp(S)$, and write $\comp(S)=\comp^{(\min S)}(S)$. We put $\ccomp(\varnothing)=\varnothing$.
\end{definition}

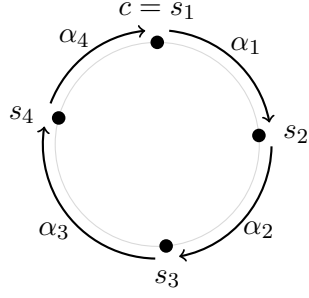
\begin{figure}
	\centering
	\begin{tikzpicture}[scale=1.35]
	\draw[gray!30] (0,0) circle (1);
	\foreach \a in {90, 5, -85, 165}{
			\node[circle,fill=black,inner sep=1.8pt] at (\a:1) {}; }
	\node[font=\small,anchor=south] at (90:1.12) { $ c=s_1 $ };
	\node[font=\small,anchor=west] at (5:1.13) { $ s_2 $ };
	\node[font=\small,anchor=north] at (-85:1.13) { $ s_3 $ };
	\node[font=\small,anchor=east] at (165:1.13) { $ s_4 $ };
	\draw[->,thick] (84:1.12) arc (84:11:1.12);
	\node[font=\small] at (48:1.30) { $ \alpha_1 $ };
	\draw[->,thick] (-1:1.12) arc (-1:-79:1.12);
	\node[font=\small] at (-40:1.30) { $ \alpha_2 $ };
	\draw[->,thick] (-91:1.12) arc (-91:-189:1.12);
	\node[font=\small] at (-140:1.32) { $ \alpha_3 $ };
	\draw[->,thick] (159:1.12) arc (159:96:1.12);
	\node[font=\small] at (128:1.32) { $ \alpha_4 $ };
	\end{tikzpicture}
	\caption{Reading the gaps of $S=\{s_1<s_2<s_3<s_4\}\subseteq\Z/n\Z $ clockwise from $ c=s_1 $, the points of $ S $ appear in the order $ s_1,s_2,s_3,s_4 $, and $ \alpha_i $ is the number of steps of the $ i $-th arc, giving $ \comp^{(c)}(S)=(\alpha_1,\alpha_2,\alpha_3,\alpha_4)$.}\label{fig:comp}
\end{figure}

Let $ w:\Z/n\Z \to [k] $ be a word of length $n$ and ordered alphabet $[k]$ for some positive integer $k$, such that $1 < \cdots < k$.
Following~\cite{magicnumber}, we define the \emph{cyclic descents} of $w$ as
\begin{equation}\label{eq:cdes}
	\cDes(w):=\{j\in\Z/n\Z:w(j)>w(j+1)\},\qquad \cdes(w) = |\cDes(w)|
\end{equation}
The \emph{content} of $ w $
is the \emph{weak} composition
\begin{equation}
	\content(w):= (|w^{-1}(1)|,\dots,|w^{-1}(k)|).
\end{equation}

\begin{definition}\label{def:hypersimplicial}
	Let $\alpha$ be a weak composition with $k$ parts.
	Viewing $\alpha: \Z/k\Z \to \Z_{\geq 0}$, we say that $\alpha \models n$ is \emph{hypersimplicial} if every positive entry satisfies $ \alpha_i > 1 $ and is
	followed by at most $ \alpha_i - 2 $ zeros on the circle.
	A cyclic word is \emph{hypersimplicial} if it has
	hypersimplicial content.
\end{definition}
Let $w$ be a cyclic word $\Z/n\Z \to [k]$. For $ p \in \Z/k\Z $, define the \emph{$p$-shift} of $ w $ as the cyclic word
\begin{equation}\label{eq:pshift}
	w^{(p)}(i) = w(i) - p \mod k.
\end{equation}

\begin{definition}\label{def:reading_word}
  Label the $k$ circular positions clockwise by $1,\ldots,k$, choosing the origin so that the block containing $1$ occupies position $k$.
	The \emph{reading word} of a decorated ordered set partition in the form of \eqref{eq:dosp} is the cyclic word $ w $ defined by
	\begin{equation}\label{eq:reading-osp}
		w_\dosp(i)= j \text{ if the block containing }i \text{ is at position }j.
	\end{equation}
\end{definition}
The reading word gives a bijection from $\hDOSP(k,n)$ to the hypersimplicial words $w:\Z/n\Z\to[k]$ with $w(1)=k$.
Let $\cDes(\dosp):= \cDes(w_{\dosp})$. Then $\wind(\dosp) = |\cDes(\dosp)|$.
The reading word of an ordered set partition $(S_1,\dots,S_p)$ is the reading word of the decorated ordered set partition $((S_1)_1,\ldots,(S_p)_1)$ with decorations all equal to 1.

\begin{example}\label{eg:readS}
	The reading word of the decorated ordered set partition in Figure~\ref{fig:dosp} is $762577226525$, with cyclic descents $\{1,2,6,9,10\}$ and content $(0,4,0,0,3,2,3)$.
\end{example}

\subsection{Main results}\label{ss:main_results}
Our first main result is a refinement of Early--Kim's theorem by giving a combinatorial formula for Stanley's flag $ h
$-numbers of the alcove triangulation complex.
\begin{theorem}\label{thm:hmain}
	Let $ \Tri $ be the triangulation complex of $ \hyp $.
Then, for each cyclic composition $[\alpha]$ of $n$, and also for $[\alpha]=\varnothing$, we have
	\begin{equation}
		\sum_{S: \ccomp(S) = [\alpha]} h_S(\Tri) = \left| \left\{ \dosp \in \hDOSP(k,n): \ccomp(\cDes(\dosp)) = [\alpha] \right\}\right|.
	\end{equation}
\end{theorem}

Our second result interprets flag $f$-numbers in terms of $S_n$-modules over $\Q$. Let $\alpha= (\alpha_1,\dots,\alpha_r) $ be a weak composition of $ n $. The \emph{Young subgroup} is $ S_\alpha \cong S_{\alpha_1}
\times \cdots \times S_{\alpha_r} $, and the \emph{Young permutation module} of $\alpha$ is $ M^\alpha \cong
\Ind_{S_\alpha}^{S_n} \triv $, the $ S_n $ -module on the ordered set partitions $ (B_1,\dots,B_r) $ of $ [n] $ with $
|B_i| = \alpha_i $. The Frobenius character of $ M^\alpha $ is $ \ch M^{\alpha} = h_\alpha $, the complete homogeneous
symmetric function. Let
\begin{equation}
	\label{eq:Mkn} \M_{k,n} = \bigoplus_{\substack{\lambda \in \Z_{\geq0}^{\Z/k\Z}, \\ \lambda \models n, \\ \lambda \text{ is hypersimplicial}}} M^\lambda.
\end{equation}
\begin{theorem}\label{thm:repmain}
	Let $ \gcd(k,n)=1 $ and let $ S $
	be a nonempty subset of $ [n] $ with $\ccomp(S) = [\alpha]$. Then
	\begin{equation}
		k \, f_S(\Tri) = \dim \Hom_{S_n}(M^{\alpha}, \M_{k,n})
	\end{equation}
\end{theorem}

\subsection*{Organization}
Section~\ref{sec:triangulation} describes the faces of the alcove triangulation of $\hyp$. In Section~\ref{sec:symm} we express the flag $f$-numbers in terms of symmetric functions, and in Section~\ref{sec:hypwords} we prove Theorem~\ref{thm:hmain}.

\subsection*{Acknowledgements}
We thank Nick Early, Elisabeth Bullock, Donghyun Kim, Lauren Williams, and Sylvie Corteel for helpful discussion.

\section{The alcove triangulation}\label{sec:triangulation}

We introduce the alcove triangulation of the hypersimplex. Lam and Postnikov~\cite{alcove1} showed that the maximal
faces of $ \Tri $ are parametrized by minimal \emph{circuits}. Parisi, Sherman-Bennett, Tessler, and
Williams~\cite{magicnumber} showed that the circuits can be described as the cyclic descents set of permutations.

\begin{definition}[\cite{alcove1}]
	A sequence $ (I_1,\dots,I_n) $ of $ k $ -subsets of $ [n] $
	is a \emph{circuit} if for each $ i \in \Z/n\Z $, there exists $ j \in \Z/n\Z $ such that $j \in I_i$ and $j+1 \notin I_i$ and $I_{i+1} = I_i \setminus \{j\} \cup \{j+1\} $.
        Here the site labels are read modulo $n$.
\end{definition}

\begin{theorem}[\cite{alcove1,magicnumber}]\label{thm:max_face}
  The maximal faces of $\Tri$ are precisely the sets of vertices $\{I_p:p\in\Z/n\Z\}$ obtained from permutations $w\in S_n$ with $\cdes(w)=k$. In circuit order the vertices satisfy
\begin{equation}\label{eq:circuit}
  I_p=(I_{p-1}\setminus\{w^{-1}(p)-1\})\cup\{w^{-1}(p)\},
\end{equation}
where
\begin{equation}\label{eq:cDescirc}
  I_p=\cDes(w^{(p)}).
\end{equation}
A cyclic shift of the values of $w$ changes the starting vertex of this description. In particular, each maximal face has a representative with $w(1)=n$.
\end{theorem}

\begin{table}
	\begin{tabular}{|c|c|}
		\hline
		$w$  & $(I_0,\dots,I_3)$ \\
		\hline
		4312 & (12,13,14,24)     \\
		\hline
		4213 & (12,13,23,24)     \\
		\hline
		4231 & (13,14,24,34)     \\
		\hline
		4132 & (13,23,24,34)     \\
		\hline
	\end{tabular}
	\caption{The permutations and vertex sets indexing simplices in the hypersimplex $\Delta_{2,4}$.}\label{tab:d24}
\end{table}

\begin{example}\label{ex:maxface}
	In Figure~\ref{fig:d24}, the four simplices triangulating the hypersimplex $\Delta_{2,4}$ are shown in Table~\ref{tab:d24}.
	Take the second row in that table where $w^{(0)} = 4213$ with $\cDes(w^{(0)}) = \{1,2\}$. We have
	\begin{align*}
		w^{(1)} & = 3142, & \cDes(w^{(1)}) = \{1,3\}, \\
		w^{(2)} & = 2431, & \cDes(w^{(2)}) = \{2,3\}, \\
		w^{(3)} & = 1324, & \cDes(w^{(3)}) = \{2,4\}.
	\end{align*}
\end{example}

\begin{theorem}\label{thm:coloring}
The coloring $\col(I):=\sum_{i\in I}i\in\Z/n\Z$ makes $\Tri$ completely balanced.
\end{theorem}

\begin{proof}
	Suppose $F$ is a maximal face with circuits $(I_0,\dots,I_{n-1})$, then \[\col(I_{p}) \equiv \col(I_{p-1}) + 1 \bmod n, \qquad \text{ for any }p \in \Z/n\Z, \] so all the vertices of $F$ are of different colors, and each color appears only once. Any lower dimensional face is
    contained in a maximal face, and this property is inherited.
\end{proof}

Let $\Z/n\Z$ act on subsets $S\subseteq[n]$ by
\begin{equation}
	S + j:= \{s + j \bmod n \colon s \in S\}, \qquad \text{ for each } j \in \Z/n\Z.
\end{equation}
\begin{proposition}
	Let $d = \gcd(k,n)$. For every $S \subseteq \Z/n\Z$, the flag numbers satisfy $f_S(\Tri) = f_{S+k}(\Tri)$ and $h_S(\Tri) = h_{S+k}(\Tri)$.
\end{proposition}
\begin{proof}
	Consider the rotation of the $n$-coordinates
	\[
		\rho: \R^n \to \R^n, (x_1,\dots,x_n) \mapsto
		(x_2,\dots,x_{n},x_1)
	\]
	which preserves the alcove triangulation. Then $\typ(\rho F) = \typ(F) - k$ and $f_S(\Tri) = f_{S+k}(\Tri)$.
	Applying M\"obius transformation, we obtain the same equality for $h_S$.
\end{proof}

As $k$ generates a subgroup of $\Z/n\Z$ with index $\gcd(k,n)$, we have the following corollary.
\begin{corollary}\label{cor:gcd}
Let $d=\gcd(k,n)$. The functions $j\mapsto f_{S+j}(\Tri)$ and $j\mapsto h_{S+j}(\Tri)$ depend only on $j\bmod d$ and are constant if $d=1$.
\end{corollary}

For any cyclic word $w$ on $\Z/n\Z$, let
\begin{equation}
	\cEq(w)  =\{j\in\Z/n\Z:w(j)=w(j+1)\}, \qquad \ceq(w) = |\cEq(w)|,
\end{equation}
denote its set of \emph{cyclic equalities}.

\begin{lemma}\label{lem:alcove_to_word}
	Let  $ F \in \Tri $ be a face of dimension $ p-1 $,
	\begin{enumerate}
        \item Order the vertices of $F=(I_0,\ldots,I_{p-1})$ lexicographically, and put $\base(F)=I_0$. Identify $I_t$ with its indicator vector and write $\alpha_j=e_j-e_{j-1}$ for $2\le j\le n$.
          For $1\le t<p$, there is a unique nonempty $A_t\subseteq[2,n]$ such that $I_t-I_{t-1}=\alpha_{A_t}:=\sum_{j\in A_t}\alpha_j$.
	\item The $ A_i $'s are mutually disjoint. Let $ A_{p}:= [n] \setminus \left(\bigcup_{i=1}^{p-1} A_i\right) $.
	\item Let $ w_F: [n] \to [p] $ be the word defined by
	\begin{equation}\label{eq:wF}
			w_F(i) = \min\{t: i \in A_1 \cup \cdots \cup A_t\},
	\end{equation}
	then
	\begin{equation}\label{eq:Fdeseq}
          \cDes(w_F) \subseteq \base(F) \subseteq \cDes(w_F) \cup \cEq(w_F).
	\end{equation}
	Conversely, for any surjective word $w:\Z/n\Z\to[p]$ with $w(1)=p$ and any $k$-subset $I\subseteq[n]$ such that
	\begin{equation}\label{eq:deseq}
          \cDes(w) \subseteq I \subseteq \cDes(w) \cup \cEq(w),
	\end{equation}
	there is a unique face $ F = (I = I_0,\dots,I_{p-1}) $ with $\base(F) = I_0 = I$ and $w_F = w$.
	\item For any word $w:\Z/n\Z\to[p]$, the number of $I\in\binom{[n]}{k}$
          satisfying~\eqref{eq:deseq} is
          \[
            \binom{\ceq(w)}{k - \cdes(w)}.
          \]
	\end{enumerate}
\end{lemma}
\begin{proof}
\noindent
	\begin{enumerate}
		\item\label{item1}
		Suppose that $ F=(I_0,\dots,I_{n-1}) $
		is a maximal face, and $w$ is a permutation taken from Theorem~\ref{thm:max_face} with $w(1)=n$.
		By~\eqref{eq:circuit}, we see that $I_p$ is lexicographically greater than $I_{p-1}$ for $p \geq 1$.
		If $ F $ is a face of lower dimension, then we may embed it into a maximal face, and the conclusion follows from the case for maximal faces.
		\item For a maximal face $ F $ with $ w \in S_n $ as in Theorem~\ref{thm:max_face}, we have $ A_i =
		\{w^{-1}(i)\} $. They are mutually disjoint because $ w $ is a permutation. For lower dimensional faces, we
		embed them into maximal faces to conclude.
		\item Suppose that $ F $ is a maximal face, and $ w \in S_n $ is as in Theorem~\ref{thm:max_face}. Moreover,
		by~(\ref{item1}), we may assume that $ w(1) = n $. Then $ w_F = w $ and $ \base(F) = \cDes(w) $
		by~\eqref{eq:cDescirc}. If $ F $ is lower dimensional, we may embed $ F $ into a maximal face $ F' $ with a
		permutation $ w $ with $ w(1) = n $ as in Theorem~\ref{thm:max_face} and~(\ref{item1}). Consequently, $
		\base(F') = \cDes(w) $. As $ F \subseteq F' $, the tuple $ (A_1,\dots,A_p) $ is an ordered set partition of
		which $ w $ is a refinement, and $ w_F $ is the reading word~\eqref{eq:reading-osp} of $ (A_1,\dots,A_p) $. The
		set $ \cDes(w_F) $ consists of indices $ i $ that appear in a later block than $ i+1 $, and the set $ \cEq(w_F)
		$ consists of indices $ i $ that appear in the same block as $ i+1 $. As $ w $ is a refinement of $
		(A_1,\dots,A_p) $, the conclusion follows. Conversely, given the word $ w: \Z/n\Z \to \Z/p\Z$, let $A_i = w^{-1}(i)$. We can form $ (I_1,\dots,I_{p}) $ inductively by
\[
  I_1 :=  I, \qquad I_p = I_{p-1}+\alpha_{A_p} .
\]
To see that $F$ is a face, we may find any permutation $v$ such that $v(i) < v(j)$ if $w(i) < w(j)$ in the natural order of $[p]$. This embeds $F$ into a maximal face.
		\item We may assume that $ \cdes(w) \leq k $. The condition on $ I $ means that $ I \setminus \cDes(w) $ can be
		any subset of $ \cEq(w) $ with cardinality $ k-\cdes(w) $, hence the conclusion.
	\end{enumerate}
\end{proof}
\begin{example}
	Choose $F = (I_0,I_1)$ such that $I_0 = \{1,2\}, I_1 = \{2,4\}$.
	We have $I_1 - I_0 = e_4-e_1 = \alpha_2 + \alpha_3 + \alpha_4$ so $A_1 = \{2,3,4\}$ and $A_2 = \{1\}$.
	The reading word of this ordered set partition is $w_F = 2111$.
	We have $\cDes(w_F) = \{1\}$ and $\cEq(w_F) = \{2,3\}$.
\end{example}

For a nonempty face $F\in\Tri$, let
\begin{equation}
  \comp(F):=\comp^{(\col(\base(F)))}(\typ(F)).
\end{equation}
\begin{lemma}\label{lem:compF}
For a nonempty face $F\in\Tri$, let $w_F$ be the word in (3) of Lemma~\ref{lem:alcove_to_word}. Then
	\[
		\comp(F) = \content(w_F).
	\]
\end{lemma}
\begin{proof}
  For $1\le t<p$, the difference $I_t-I_{t-1}=\alpha_{A_t}$ gives $\col(I_t)-\col(I_{t-1})=|A_t|$ in $\Z/n\Z$.
  Since the $A_t$ partition $[n]$, the last arc back to $\col(I_0)$ has length $|A_p|$. This also covers $p=1$, when that arc has length $n$. The successive arc lengths are therefore $(|A_1|,\ldots,|A_p|)=\content(w_F)$ by~\eqref{eq:wF}.
\end{proof}
For a composition $\alpha$, let $W(\alpha)$ be the set of cyclic words of content $\alpha$.

\begin{proposition}\label{prop:compF}
	Let $ \alpha \models n $
	be a composition with $ p $ parts. The number of nonempty faces $ F \in \Tri $ with $ \comp(F) = \alpha $ is equal to
	\begin{equation}\label{eq:Nalpha}
		\sum_{\substack{w \in W(\alpha) \\ w(1) = p}} \binom{\ceq(w)}{k - \cdes(w)}.
	\end{equation}
\end{proposition}

\begin{proof}
	Let $F \in \Tri$ be a nonempty face with $\comp(F) = \alpha$ and $w_F$ be as in (3) of Lemma~\ref{lem:alcove_to_word}. By Lemma~\ref{lem:compF}, we have $w_F \in W(\alpha)$.
	The definition of $\base(F)$ forces $w_F(1) = p$.
	For each such word $w$, Lemma~\ref{lem:alcove_to_word} (4) shows that there are $\binom{\ceq(w)}{k - \cdes(w)}$ possible choices of $\base(F)$.
\end{proof}

\begin{corollary}
	Let $ [\alpha] $ be a cyclic composition. Then
	\begin{equation}
		\sum_{\substack{S \subseteq [n] \\ \ccomp(S) = [\alpha]}} f_S(\Tri) = \sum_{\beta \in [\alpha]} \sum_{\substack{w \in W(\beta) \\ w(1) = p}} \binom{\ceq(w)}{k - \cdes(w)}.
	\end{equation}
\end{corollary}
\begin{proof}
	Summing over all $\comp(F) = \beta \in [\alpha]$ in Proposition~\ref{prop:compF}, we obtain the desired equality.
\end{proof}

\begin{theorem}
	For all $ 0 < k < n $
	and every nonempty subset $ S \subseteq[n] $ with $ |S| = p $, over $\Q[q]/(q^n-1) $, we have
	\begin{equation}
		f_S(\Tri) = \sum_{i \in S} [q^i] \sum_{\substack{\content(w) = \comp^{(i)}(S) \\ w(1)=p}} q^{\cmaj(w)} e_{k-\cdes(w)}\left( q^j: j \in \cEq(w) \right)
	\end{equation}
	where
	\begin{equation}
		\cmaj(w):= \sum_{i \in \cDes(w)} i
	\end{equation}
	is the \emph{cyclic major index}. Here $[q^i]$ takes the coefficient with exponent congruent to $i$ modulo $n$ after reduction modulo $q^n-1$. We use $e_r=0$ when $r<0$ or $r$ exceeds the number of variables.
\end{theorem}
\begin{proof}
	Let $ F $ be a face with $ \typ(F) = S $, and let $ w_F $
	be as in Lemma~\ref{lem:alcove_to_word}. Since $ \typ(F) = S $, there exists $ i \in S $ such that $ \col(\base(F)) = i $.
	By Lemma~\ref{lem:alcove_to_word}, $
		\cDes(w_F) \subseteq \base(F) \subseteq \cDes(w_F) \cup \cEq(w_F) $
, so the inner sum is a generating function of $
		q^{\sum_{j \in \base(F)} j} $.
\end{proof}

\begin{lemma}\label{lem:alphap}
	Let $ \alpha = (\alpha_1,\dots,\alpha_p) \models n $. Then
	\begin{equation}
		\alpha_p \sum_{w \in W(\alpha)} \binom{\ceq(w)}{k-\cdes(w)} = n \sum_{\substack{w \in W(\alpha) \\ w(1) = p}} \binom{\ceq(w)}{k - \cdes(w)}.
	\end{equation}
\end{lemma}
\begin{proof}
	The $\Z/n\Z$ action on cyclic words generated by
	\begin{equation}\label{eq:nZaction}
		w \mapsto w', \qquad w'(i) = w(i-1),
	\end{equation}
	preserves $\cdes$ and $\ceq$.
	There are $\alpha_p$ many $i$ such that $w(i) = p$, so the left hand side counts the number of pairs $(w,i)$ for which $w\in W(\alpha)$ and $w(i) = p$, weighted by the binomial coefficient.
	The right hand side applies the $\Z/n\Z$ action to $w$ with $w(1) = p$ and counts the same set of pairs with the same weight.
\end{proof}

\section{Symmetric function identities}\label{sec:symm}

We express the flag $f$-numbers in terms of a generating function for cyclic descents and equalities. They occur as a coefficient of the circular Smirnov word enumerator studied in~\cite{ellzey,EllzeyWachs}.
Let
\begin{equation}\label{eq:Gkn}
	G_{k,n} = \sum_{w: \Z/n\Z \to \Z_{>0}} \binom{\ceq(w)}{k - \cdes(w)} \prod_{i \in \Z/n\Z} x_{w(i)}
\end{equation}
be a power series in $x_1,x_2,\dots$.

Let
\begin{equation}\label{eq:Hypkn}
	\Hyp_{k,n}= \ch \mathcal M_{k,n}
\end{equation}
where $M_{k,n}$ is defined in \eqref{eq:Mkn}.
In calculations with $\Theta$ and $X$ we specialize to
$x_1,\ldots,x_p$ by setting $x_i=0$ for $i>p$. Let $\Theta$ be the $p\times p$ matrix with entries
\begin{equation}
	\Theta_{ij}=
	\begin{cases}
		1,   & i<j, \\
		1+z, & i=j, \\
		z,   & i>j,
	\end{cases}
\end{equation}
and let $ X $ be the diagonal matrix with $ x_1,\dots,x_p $ along the diagonal.

\begin{theorem}\label{thm:GHyp}
	For all $n \geq 2$ and $1 \leq k \leq n-1$, we have
	\begin{equation}
		G_{k,n}= \frac{n}{k} \Hyp_{k,n}
	\end{equation}
\end{theorem}

\begin{lemma}\label{lem:Gtrace}
	\begin{equation}
		\sum_{k \geq 0} G_{k,n}(x_1,\dots,x_p) z^k = \tr((\Theta X)^n).
	\end{equation}
\end{lemma}

\begin{proof}
	Expanding the trace of the $ n $ -th power of $ \Theta X $, the trace is
	\begin{equation}
		\tr((\Theta X)^n) = \sum_{w: \Z/n\Z \to [p]} \prod_{i \in \Z/n\Z} \Theta_{w(i)w(i+1)} x_{w(i+1)}
	\end{equation}
	the sum being over all cyclic words of length $ n $ with values in $[p]$. If $ w(i) < w(i+1) $, then $
		\Theta_{w(i)w(i+1)} = 1 $; if $ w(i) = w(i+1) $, i.e. $ i \in \cEq(w) $, then $
		\Theta_{w(i)w(i+1)} = 1+z $; if $ w(i) > w(i+1) $, i.e. $ i \in \cDes(w) $, then $
		\Theta_{w(i)w(i+1)} = z $. Therefore,
	\begin{equation}
		\tr((\Theta X)^n) = \sum_{w: \Z/n\Z \to [p]} z^{\cdes(w)} (1+z)^{\ceq(w)} x^w = \sum_{w: \Z/n\Z \to [p]} \sum_{k \geq 0} \binom{\ceq(w)}{k-\cdes(w)} z^k x^w.
	\end{equation}
      \end{proof}
In the ring of symmetric functions, let
\begin{equation}
	H(t) = \prod_{i \geq 1} \frac{1}{1 - x_i t} = \sum_{n \geq 0} h_n t^n.
\end{equation}
Then
\begin{equation}
	\frac{1}{H(t)} = \prod_{i \geq 1} (1-x_i t) = \sum_{j \geq 0} (-1)^j e_j t^j.
\end{equation}

\begin{lemma}\label{lem:detI-t}
	In the alphabet $x_1,\ldots,x_p$, we have
	\begin{equation}
		\det(I - t \Theta X) = \frac{1}{1-z} \left( \frac{1}{H(t)} - \frac{z}{H(zt)} \right).
	\end{equation}
\end{lemma}

\begin{proof}
	Let $ \Theta^{(j)} $ be the principal $ j \times j $ -minor of $ \Theta $. For $2 \le j \le p$, subtracting row $j-1$ from row $j$ and expanding along the last row gives
	\begin{equation}
		\det \Theta^{(j)} = (1+z) \det \Theta^{(j-1)} - z \det \Theta^{(j-2)}
	\end{equation}
	with initial conditions $
		\det \Theta^{(0)} = 1 $ and $
		\det \Theta^{(1)} = 1+z $.
	By induction, we can show that $
		\det \Theta^{(j)} = [j+1]_z = 1 + z + \cdots + z^j = \frac{1 - z^{j+1}}{1-z} $.
	We expand the determinant
	\begin{equation}
		\det(I - t\Theta X) = \sum_{j \geq 0} (-t)^j \sum_{|S|=j} \det(\Theta_{S,S}) x^S.
	\end{equation}
	Then
	\begin{multline}
		\det(I - t\Theta X) = \sum_{j \geq 0} (-1)^j [j+1]_z e_j t^j \\= \frac{1}{1-z} \left( \sum_{j \geq 0} (-1)^j e_j t^j - z \sum_{j \geq 0} (-1)^j e_j (zt)^j \right) \\= \frac{1}{1-z}\left(\frac{1}{H(t)} - \frac{z}{H(zt)}\right).
	\end{multline}
\end{proof}

\begin{lemma}\label{lem:Hyp-log}
For each $k\ge1$, in the ring of symmetric functions we have
\begin{equation}
  \sum_{n>k}\Hyp_{k,n}t^n
  =k[z^k]\log\frac{1-z}{H(zt)-zH(t)}.
\end{equation}
\end{lemma}

\begin{proof}
	Compute the right hand side by
	\begin{multline}
		H(zt) - z H(t) = \sum_{n \geq 0} h_n (zt)^n - z \sum_{n \geq 0} h_n t^n \\ = h_0(1-z) + z \sum_{n \geq 2} (z^{n-1} - 1) h_n t^n = (1-z) - z(1-z) \sum_{n \geq 2} [n-1]_z h_n t^n.
	\end{multline}
	Therefore,
	\begin{equation}
		\log \frac{1-z}{H(zt) - z H(t)} = \sum_{p \geq 1} \frac{1}{p} \left(z \sum_{n \geq 2} [n-1]_z h_n t^n \right)^p.
	\end{equation}
	Observe that
	\begin{equation}
		z \sum_{n \geq 2} [n-1]_z h_n t^n = \sum_{j \geq 1} z^j \sum_{n \geq j+1} h_n t^n.
	\end{equation}
	The coefficient of $ z^k $ in $
		\left(\sum_{j \geq 1} z^j \sum_{n \geq j+1} h_n t^n \right)^p $ is
	\begin{equation}\label{eq:pair_hyp_cont}
		\sum_{\mu} \sum_\lambda h_\lambda t^{|\lambda|}.
	\end{equation}
	Here $ \mu = (\mu_1,\dots,\mu_p) $ is a composition of $ k $, and $ \lambda = (\lambda_1,\dots,\lambda_p)$ is a composition such that $\lambda_i \geq \mu_i + 1 $ for every $ i \in \Z/p\Z $.
        To see the relation with Definition~\ref{def:hypersimplicial}, choose a starting position in $\Z/k\Z$ and place $\lambda_1$ there. Proceed clockwise, placing $\lambda_{i+1}$ at distance $\mu_i$ from $\lambda_i$, with indices taken modulo $p$.
        Thus exactly $\mu_i-1$ zeros occur after $\lambda_i$ and before the next positive entry. By Definition~\ref{def:hypersimplicial}, hypersimpliciality requires $\mu_i-1 \leq \lambda_i-2$, equivalently $\lambda_i \geq \mu_i+1$.
        This construction gives a hypersimplicial weak composition $\alpha$ with a distinguished positive position. Conversely, reading clockwise from that position recovers $\lambda$ as the sequence of positive entries and $\mu$ as the sequence of distances between them.
        There are $k$ choices of starting position for each pair of $(\lambda,\mu)$, and each resulting $\alpha$ has $p$ choices of distinguished position.
        Since $h_\alpha = h_\lambda$, multiplying by $k/p$ counts each $\alpha$ with weight $h_\alpha t^{|\alpha|}$. Summing over $p$ proves the lemma.
\end{proof}

\begin{proof}[Proof of Theorem~\ref{thm:GHyp}]
	By Lemma~\ref{lem:Gtrace}
	\begin{equation}
		-\log \det(I - t\Theta X) = \sum_{n \geq 1} \frac{\tr((\Theta X)^n)t^n}{n} = \sum_{n \geq 1} \sum_{k \geq 0} \frac{G_{k,n} z^k t^n}{n}
	\end{equation}
	By Lemma~\ref{lem:detI-t},
	\begin{align}
		-\log \det(I - t\Theta X) & = \log H(t) + \log H(zt) + \log \frac{1-z}{H(zt) - zH(t)} \\
		                          & = \log H(t)+\log H(zt)+\sum_{n\ge2}\sum_{k=1}^{n-1}\frac{\Hyp_{k,n}t^nz^k}{k}.
	\end{align}
The monomials in $\log H(t)$ have $z$-degree zero, and those in $\log H(zt)$ have equal $z$- and $t$-degrees. For $1\le k<n$, comparison of the coefficients of $z^kt^n$ gives $G_{k,n}=\frac{n}{k}\Hyp_{k,n}$ in the alphabet
$x_1,\ldots,x_p$. Since this holds for every $p$, the identity follows in the ring of symmetric functions.
\end{proof}

\begin{corollary}\label{cor:Gsym}
	The sum
	\[
		\sum_{w \in W(\alpha)} \binom{\ceq(w)}{k - \cdes(w)} = \langle h_\alpha, G_{k,n} \rangle
	\]
	depends only on the multiset of parts of $ \alpha $.
\end{corollary}
\begin{proof}
	By definition, $\Hyp_{k,n}$ is a symmetric function. By Theorem~\ref{thm:GHyp}, $G_{k,n}$ is also a symmetric function. The definition of $G_{k,n}$ expands it into the monomial basis $\{m_\alpha\}$ of the ring of symmetric functions. Since the monomial symmetric function is dual to the complete homogeneous symmetric function, taking inner product with $h_\alpha$ extracts the coefficient of $m_\alpha$, which is the left hand side.
\end{proof}

\begin{theorem}\label{thm:ftoG}
	For all $ 0 < k < n $
	and every nonempty subset $ S \subseteq [n] $ with $
		\ccomp(S) = [\alpha]$, we have
	\begin{equation}
		\sum_{j \in \Z/n\Z} f_{S+j}(\Tri) = \langle h_\alpha, G_{k,n} \rangle.
	\end{equation}
\end{theorem}
\begin{proof}
	Let
	\[
		N(\alpha) = \sum_{\substack{w \in W(\alpha) \\ w(1) = p}} \binom{\ceq(w)}{k - \cdes(w)}.
              \]
              A face $F$ of type $S+j$ determines a unique $s\in S$ with $s+j=\col(\base(F))$, so $\comp(F) = \comp^{(s)}(S)$.
              Conversely, for each $s\in S$, a face with this composition has type $S+j$ for the unique $j=\col(\base(F))-s$.
              Counting these pairs by Proposition~\ref{prop:compF} gives the first equality below,
              and Lemma~\ref{lem:alphap} gives the second.
	\begin{multline}
  \sum_{j\in\Z/n\Z} f_{S+j}(\Tri)=\sum_{i=1}^{p}N(\alpha^{(i)})
  \\=\frac1n\sum_{i=1}^{p}(\alpha^{(i)})_p
    \sum_{w\in W(\alpha^{(i)})}\binom{\ceq(w)}{k-\cdes(w)}
  =\langle h_\alpha,G_{k,n}\rangle .
\end{multline}
The final equality follows from Corollary~\ref{cor:Gsym} and
$\sum_{i=1}^p(\alpha^{(i)})_p=n$.
\end{proof}

\begin{proof}[Proof of Theorem~\ref{thm:repmain}]
	If $\gcd(k,n) = 1$, by Corollary~\ref{cor:gcd}, the left hand side of Theorem~\ref{thm:ftoG} is equal to $n f_S(\Tri)$. By Theorem~\ref{thm:GHyp}, the right hand side is equal to $\frac{n}{k} \langle h_\alpha, \Hyp_{k,n} \rangle$.
The Frobenius characteristics of $M^\alpha$ and $\M_{k,n}$ are $h_\alpha$ and $\Hyp_{k,n}$, so their Hall inner product is $\dim\Hom_{S_n}(M^\alpha,\M_{k,n})$.
\end{proof}

By Corollary~\ref{cor:gcd}, we have the following fact.
\begin{corollary}
	Let $ d = \gcd(k,n) $, and let $ S $
	be a nonempty subset of $ [n] $ with $
		\ccomp(S) = [\alpha] $. Then $ f_{S+j}(\Tri) $ depends on $ j $ only through $ j \bmod d $, and
	\begin{equation}
		\frac{n}{d} \sum_{j=0}^{d-1} f_{S+j}(\Tri) = \sum_{w \in W(\alpha)} \binom{\ceq(w)}{k - \cdes(w)}.
	\end{equation}
\end{corollary}

\section{Hypersimplicial words}\label{sec:hypwords}

For $S\subseteq[n]$, let $W_S$ be the set of hypersimplicial words $w:\Z/n\Z\to[k]$ with $\cDes(w)\subseteq S$.
When $S\ne\varnothing$, write $\alpha=\comp(S)$.

\begin{theorem}\label{thm:contingency}
For $S\ne\varnothing$, the cardinality of $W_S$ is
	\begin{equation}
		|W_S| = \langle h_\alpha, \Hyp_{k,n} \rangle.
	\end{equation}
\end{theorem}
Recall that $\langle h_\alpha, h_\lambda \rangle$ is equal to the number of matrices with nonnegative integer entries with row sum $\alpha$ and column sum $\lambda$. The set of matrices with fixed row sum vector and column sum vector are also called \emph{contingency tables}.
\begin{proof}
Write $S=\{s_1<\cdots<s_p\}$ and $s_{p+1}=s_1+n$, reading positions modulo $n$. Then $\alpha_i=s_{i+1}-s_i$, and $\cDes(w)\subseteq S$ exactly when $w$ is weakly increasing on each interval of positions $s_i+1,\ldots,s_{i+1}$. Fix a hypersimplicial content $\lambda$.
Let $a_{ij}$ count the occurrences of $j$ in the $i$-th interval.
The resulting \emph{contingency table} has row sums $\alpha$ and column sums $\lambda$. Conversely, fill the $i$-th interval in increasing order of the letters, using $a_{ij}$ copies of $j$. This recovers the unique word for that table. There are $\langle h_\alpha,h_\lambda\rangle$ such tables. Summing over $\lambda$ gives the result by~\eqref{eq:Mkn} and~\eqref{eq:Hypkn}.
\end{proof}

\begin{lemma}\label{lem:average}
For every nonempty $S\subseteq[n]$ and $p\in[k]$,
\[
\sum_{w \in W_S} |w^{-1}(p)| = \frac{n}{k} |W_S|.
\]
\end{lemma}
\begin{proof}
By the preceding bijection the left hand side is $\sum_\lambda\lambda_p\langle h_\alpha,h_\lambda\rangle$, summed over hypersimplicial contents.
Rotating the parts of $\lambda$ permutes these contents and preserves $h_\lambda$, so the sum is independent of $p$.
Summing over $p$ gives $n|W_S|$ and proves the formula.
\end{proof}

The bijection in Definition~\ref{def:reading_word} extends by letter shifts~\eqref{eq:pshift} to a bijection $(\dosp,r)\mapsto w_\dosp^{(r)}$ from $\hDOSP(k,n)\times\Z/k\Z$ to all hypersimplicial words on $[k]$.
A shift preserves every clockwise distance between consecutive letters, so $\cdes(w_\dosp^{(r)})=\wind(\dosp)$ for every $r\in\Z/k\Z$.

\begin{proof}[Proof of Theorem~\ref{thm:hmain}]
For $T\subseteq[n]$, let 
\[
A_T=|\{w\in W_T:w(1)=k\}|.
\]
Fix $S\ne\varnothing$ and let $\alpha=\comp(S)$. For $j\in\Z/n\Z$, the rotation $(j \cdot w)(i)=w(i-j)$ gives a bijection $W_S\to W_{S+j}$. Summing $A_{S+j}$ over $j$ therefore counts each $w\in W_S$ once for each occurrence of $k$. We obtain
\begin{equation}\label{eq:mainmobius}
  \begin{aligned}
    \sum_{j\in\Z/n\Z} A_{S+j}
      &=\sum_{w\in W_S}|w^{-1}(k)|=\frac{n}{k}|W_S|\\
      &=\frac{n}{k}\langle h_\alpha,\Hyp_{k,n}\rangle
       =\langle h_\alpha,G_{k,n}\rangle
       =\sum_{j\in\Z/n\Z}f_{S+j}(\Tri).
  \end{aligned}
\end{equation}
Here we used Lemma~\ref{lem:average} and Theorems~\ref{thm:contingency}, \ref{thm:GHyp}, and~\ref{thm:ftoG}.
The translation sums also agree for $S=\varnothing$, since the unique normalized word is $k^n$ and $f_\varnothing(\Tri)=A_\varnothing=1$.

Let 
\[
a_T = |\{\dosp\in\hDOSP(k,n): \cDes(\dosp)=T\}|.
\]
The reading-word bijection gives $A_S=\sum_{T\subseteq S}a_T$. M\"obius inversion of~\eqref{eq:mainmobius} gives $\sum_j a_{S+j}=\sum_j h_{S+j}(\Tri)$. The distinct translates of $S$ are precisely the subsets with cyclic composition $\ccomp(S)$, each occurring $|\Stab(S)|$ times, where $\Stab(S)=\{j\in\Z/n\Z:S+j=S\}$. Dividing by this multiplicity proves the theorem.
\end{proof}

\section*{}
\bibliographystyle{alpha} \bibliography{main}

@article {alcove1,
    AUTHOR = {Lam, Thomas and Postnikov, Alexander},
     TITLE = {Alcoved polytopes. {I}},
   JOURNAL = {Discrete Comput. Geom.},
  FJOURNAL = {Discrete \& Computational Geometry. An International Journal
              of Mathematics and Computer Science},
    VOLUME = {38},
      YEAR = {2007},
    NUMBER = {3},
     PAGES = {453--478},
      ISSN = {0179-5376,1432-0444},
   MRCLASS = {52B12 (52B40)},
  MRNUMBER = {2352704},
MRREVIEWER = {Jes\'us\ A.\ De Loera},
       DOI = {10.1007/s00454-006-1294-3},
}

@article {stanley80,
    AUTHOR = {Stanley, Richard P.},
     TITLE = {Decompositions of rational convex polytopes},
   JOURNAL = {Ann. Discrete Math.},
  FJOURNAL = {Annals of Discrete Mathematics},
    VOLUME = {6},
      YEAR = {1980},
     PAGES = {333--342},
   MRCLASS = {52A43},
  MRNUMBER = {593545},
MRREVIEWER = {P.\ McMullen},
}

@misc{Early2017,
      title={Conjectures for Ehrhart $h^*$-vectors of Hypersimplices and Dilated Simplices}, 
      author={Nick Early},
      year={2017},
      eprint={1710.09507},
      archivePrefix={arXiv},
      primaryClass={math.CO},
      url={https://arxiv.org/abs/1710.09507}, 
}

@article{Kim2020,
    AUTHOR = {Kim, Donghyun},
     TITLE = {A combinatorial formula for the {E}hrhart {$h^\ast$}-vector of
              the hypersimplex},
   JOURNAL = {J. Combin. Theory Ser. A},
  FJOURNAL = {Journal of Combinatorial Theory. Series A},
    VOLUME = {173},
      YEAR = {2020},
     PAGES = {105213, 15},
      ISSN = {0097-3165,1096-0899},
   MRCLASS = {05E45 (05E10 52B20)},
  MRNUMBER = {4056090},
MRREVIEWER = {Hailun\ Zheng},
       DOI = {10.1016/j.jcta.2020.105213},
       URL = {https://doi.org/10.1016/j.jcta.2020.105213},
}

@article{Stanley1979,
    AUTHOR = {Stanley, Richard P.},
     TITLE = {Balanced {C}ohen-{M}acaulay complexes},
   JOURNAL = {Trans. Amer. Math. Soc.},
  FJOURNAL = {Transactions of the American Mathematical Society},
    VOLUME = {249},
      YEAR = {1979},
    NUMBER = {1},
     PAGES = {139--157},
      ISSN = {0002-9947,1088-6850},
   MRCLASS = {05A99 (06A10 13H10 52A40 57Q05)},
  MRNUMBER = {526314},
MRREVIEWER = {G.\ O. H. Katona},
       DOI = {10.2307/1998915},
       URL = {https://doi.org/10.2307/1998915},
}

@misc{EarlyPlates,
      title={Combinatorics and Representation Theory for Generalized Permutohedra I: Simplicial Plates}, 
      author={Nick Early},
      year={2016},
      eprint={1611.06640},
      archivePrefix={arXiv},
      primaryClass={math.CO},
      url={https://arxiv.org/abs/1611.06640}, 
}

@article {magicnumber,
    AUTHOR = {Parisi, Matteo and Sherman-Bennett, Melissa and Tessler, Ran
              and Williams, Lauren},
     TITLE = {The magic number conjecture for the {$m=2$} amplituhedron and
              {P}arke-{T}aylor polytopes},
   JOURNAL = {S\'em. Lothar. Combin.},
  FJOURNAL = {S\'eminaire Lotharingien de Combinatoire},
    VOLUME = {93B},
      YEAR = {2025},
     PAGES = {Art. 103, 12},
      ISSN = {1286-4889},
   MRCLASS = {52B05 (05A99 05E14)},
  MRNUMBER = {4973261},
}

@article {DerksenFink,
    AUTHOR = {Derksen, Harm and Fink, Alex},
     TITLE = {Valuative invariants for polymatroids},
   JOURNAL = {Adv. Math.},
  FJOURNAL = {Advances in Mathematics},
    VOLUME = {225},
      YEAR = {2010},
    NUMBER = {4},
     PAGES = {1840--1892},
      ISSN = {0001-8708,1090-2082},
   MRCLASS = {52B40 (05B35)},
  MRNUMBER = {2680193},
MRREVIEWER = {Jakayla\ R.\ Robbins},
       DOI = {10.1016/j.aim.2010.04.016},
       URL = {https://doi.org/10.1016/j.aim.2010.04.016},
}

@article {EllzeyWachs,
    AUTHOR = {Ellzey, Brittney and Wachs, Michelle L.},
     TITLE = {On enumerators of {S}mirnov words by descents and cyclic
              descents},
   JOURNAL = {J. Comb.},
  FJOURNAL = {Journal of Combinatorics},
    VOLUME = {11},
      YEAR = {2020},
    NUMBER = {3},
     PAGES = {413--456},
      ISSN = {2156-3527,2150-959X},
   MRCLASS = {05E05},
  MRNUMBER = {4101322},
MRREVIEWER = {Matt\ S.\ Davis},
       DOI = {10.4310/JOC.2020.v11.n3.a1},
       URL = {https://doi.org/10.4310/JOC.2020.v11.n3.a1},
}

@article{ellzey,
    AUTHOR = {Ellzey, Brittney},
     TITLE = {Chromatic quasisymmetric functions of directed graphs},
   JOURNAL = {S\'em. Lothar. Combin.},
  FJOURNAL = {S\'eminaire Lotharingien de Combinatoire},
    VOLUME = {78B},
      YEAR = {2017},
     PAGES = {Art. 74, 12},
      ISSN = {1286-4889},
   MRCLASS = {05E05 (05C15 05C20)},
  MRNUMBER = {3678656},
}

@phdthesis{Early2016thesis,
  author = {Early, Nicholas},
  title  = {Combinatorics and Representation Theory for Generalized Permutohedra},
  school = {The Pennsylvania State University},
  year   = {2016}
}
\end{document}